\documentclass[a4paper,11pt]{amsart}

\usepackage[a4paper,margin=1in]{geometry}
\usepackage[english]{babel}
\usepackage{amsmath}
\usepackage{amssymb}
\usepackage{multirow}
\usepackage{amsfonts}
\usepackage{amsthm}
\usepackage{indentfirst}
\usepackage{bbm}
\usepackage{color}
\usepackage{bigints}
\usepackage{graphicx}
\usepackage{booktabs}
\usepackage{array}

\newlength{\tblwidth}
\newcolumntype{L}{l}

\newcommand{\dx}{\,dx}
\newtheorem{theorem}{Theorem}[section]

\newtheorem{proposition}[theorem]{Proposition}

\newtheorem{remark}{Remark}[section]

\newtheorem{definition}{Definition}[section]

\title[Anti-Gauss Lagrange interpolation]{Anti-Gauss Lagrange interpolation: Christoffel-Darboux form, barycentric representation, and orthogonal expansion}

\author[P. D\'iaz de Alba]{P. {D\'iaz de Alba}$^{*}$}
\address{Department of Mathematics and Computer Science, University of Cagliari, Via Ospedale 72, 09124, Cagliari, Italy}
\email{patricia.diazda@unica.it}

\author[L. Fermo]{L. Fermo$^{*}$}
\address{Department of Mathematics and Computer Science, University of Cagliari, Via Ospedale 72, 09124, Cagliari, Italy}
\email{fermo@unica.it}

\author[V. Loi]{V. Loi$^{\dagger}$}
\address{Department of Science and High Technology, Insubria University, Via Valleggio 11, 22100, Como, Italy}
\email{vloi@studenti.uninsubria.it}

\makeatletter
\def\@setaddresses{}
\makeatother

\begin{document}

\begin{abstract}
The paper deals with new formulations of a Lagrange interpolant polynomial based on the nodes of the well-known anti-Gauss rule. A first representation is given in terms of the classical Christoffel-Darboux kernel appropriately modified. The second one closely follows the barycentric form of the classical Lagrange polynomial, while the third formulation represents the interpolant as a combination of an orthonormal family of polynomials with respect to the discrete anti-Gauss inner product. A numerical test shows the performance of the explored forms.
\end{abstract}


\maketitle
\begingroup
\renewcommand{\thefootnote}{\fnsymbol{footnote}}
\footnotetext[1]{Department of Mathematics and Computer Science, University of Cagliari, Via Ospedale 72, 09124, Cagliari, Italy (\texttt{patricia.diazda@unica.it}; \texttt{fermo@unica.it})}
\footnotetext[2]{Department of Science and High Technology, Insubria University, Via Valleggio 11, 22100, Como, Italy (\texttt{vloi@studenti.uninsubria.it})}
\endgroup

\section{Introduction}
The Lagrange interpolation of a given function $f$ at a set of nodes $\{x_j\}_{j=1}^n$  has attracted the attention of many researchers for both theoretical and computational reasons \cite{gautschi2012}. 
Theoretically,  the central issue concerns convergence, which is strongly influenced by the so-called Lebesgue constants, the best behaviour of which is logarithmic divergence \cite{faber1914,vertesi1990}. Consequently, considerable effort has been devoted to identify sets of interpolation nodes that guarantee such behaviour \cite{MMlibro}. Several studies \cite{MMlibro} have shown that, in specific function spaces, good candidates  are zeros of orthogonal polynomials $\{p_n\}_n$ with respect to a weight function $w$, i.e.,\begin{equation*}
\langle p_n, p_m \rangle_w =\int_{a}^b p_n(x) p_m(x) w(x) \dx =
\begin{cases}
0, & n \neq m, \\
c_n>0, & n=m.
\end{cases}
\end{equation*}

In terms of numerical computation, a new formulation of the Lagrange polynomial, i.e. the barycentric form, was developed in the late 1940s. Its advantages, robustness, and numerical stability have been emphasised in \cite{berrut2004,higham2004}. When the interpolation nodes coincide with the zeros of orthogonal polynomials, a more stable formulation can be obtained through the Christoffel-Darboux kernel, which plays an important role in the theory of orthogonal polynomials and approximation of functions; see, for instance \cite{nevai1986}. 

In 1996, D. Laurie introduced for the first time new polynomials $\tilde{p}_{n+1}$ of degree $n+1$ \cite{Laurie1} defined in terms of the classical monic orthogonal polynomials $\{p_n\}$ 
\begin{equation}\label{eq:def-ptilde}
\tilde{p}_{n+1}(x):=2p_{n+1}(x)+(a_n-x) p_{n}(x), \qquad a_n =\frac{\langle x p_n,p_n\rangle_w}{\langle p_n,p_n\rangle_w}.
\end{equation}
In that paper, he studied these polynomials proving that their zeros $\tilde{x}_k$ are real, distintic, belong to the interval $[a,b]$ for a large class of weight functions $w$, and interlace the zeros of $p_n$. These polynomials were employed in the context of the numerical integration. In particular, Laurie developed the so-called anti-Gauss quadrature rule whose $n+1$ nodes are the zeros $\tilde{x}_k$ of the polynomial given in \eqref{eq:def-ptilde}. This quadrature formula is of considerable interest for several reasons, ranging from estimating the error of the classical Gauss rule to construct more accurate schemes. Consequently, it has been generalized and extended in various directions \cite{DFR2026,Djukic2023,PranicReichel,ReichelSpalevic2021,
ReichelSpalevic2026,Spalevic2017}, and applied in several contexts  
\cite{DFR2020,DFR2025,djukic2025,Fermo2025,FRRS2024,ReichelSpalevic2022}.

Very recently, polynomials \eqref{eq:def-ptilde} were carefully analysed to assess whether their zeros provide a suitable set of nodes  for creating optimal Lagrange type interpolation processes on the real semi-axes \cite{FO2026}. Motivated by this result, and by virtue of new current studies on bounded domains, the aim of this paper is to provide new formulations for the Lagrange polynomial based on the zeros of $\tilde{p}_{n+1}$, here after named \textit{anti-Gauss Lagrange polynomial}
\begin{equation}\label{eq:lagrange}
\tilde{L}_{n+1}(f,x):=\sum_{k=1}^{n+1}\tilde{\ell}_k(x)f(\tilde{x}_k), \qquad \tilde{\ell}_k(x)= \prod_{\substack{j=1 \\ j\neq k}}^{n+1} \frac{x-\tilde{x}_j}{\tilde{x}_k-\tilde{x}_j}=\frac{\tilde{p}_{n+1}(x)}{\tilde{p}_{n+1}'(\tilde{x}_k)(x-\tilde{x}_k)}.
\end{equation} 

The main goal of Section 
\ref{sec:Darboux} is to provide a modified Christoffel-Darboux kernel adapted to our polynomial and, based on it, to obtain a first stable representation of the Lagrange polynomial \eqref{eq:lagrange}. Section \ref{sec:barycentric} is devoted to the derivation and analysis of a barycentric form of the interpolant. 
In Section \ref{sec:modal}, we also provide an orthogonal expansion. We emphasize that, in this last case, with respect to the Lagrange polynomial base on the zeros of the classical orthogonal polynomials, the main novelty lies in the proof of this representation. In fact,  the exactness of the Gauss or anti-Gauss quadrature rule cannot be exploited for the expression of the related coefficients \cite[Chapter 4]{MMlibro}. Finally, in Section \ref{sec:tests}, we provide a numerical example showing the performance of the proposed formulations.

\section{Anti-Gauss Lagrange polynomial in a Christoffel-Darboux  type form}\label{sec:Darboux}
The aim of this section is to prove that the fundamental Lagrange polynomial $\tilde{\ell}_k$ can be expressed in terms of the Christoffel numbers of the anti-Gauss formula \cite{Laurie1} and a suitably modified classical Christoffel--Darboux kernel \cite[Chapter 2]{MMlibro}. 
As a consequence, we derive a first, stable formulation of the Lagrange polynomial; see Theorem \ref{teo:darb}. 
\begin{definition}\label{def:modK}
Let $\{p_n(w,x)\}_n$ the sequence of orthonormal polynomials with respect to $w$,  $p_n(w,x)=\gamma_n x^n+\ldots$, $\gamma_n>0$.  We define the modified Christoffel-Darboux kernel as
\begin{equation}\label{Ktilde}
\tilde{K}_n(w,x,y):=K_n(w,x,y)-\frac{1}{2} p_n(w,x) p_n(w,y), \quad K_n(w,x,y)=\sum_{j=0}^{n} p_j(w,x) p_j(w,y), \qquad n \geq 0.
\end{equation}
\end{definition}
\begin{proposition} 
For every $x \neq y$, the modified Darboux kernel \eqref{Ktilde} can be expressed as
\begin{equation}\label{eq:Ktilde}
\tilde{K}_n(w,x,y)
=
\frac{p_n(y)\tilde{p}_{n+1}(x)-p_n(x)\tilde{p}_{n+1}(y)}{2c_n(x-y)}, \quad n \geq 0.
\end{equation}
\end{proposition}
\begin{proof}
From \cite[Theorem 2.2.3]{MMlibro}, one has
\begin{equation}\label{eq:Kn2}
K_n(w,x,y) = \frac{\gamma_{n}}{\gamma_{n+1}}
\frac{p_{n+1}(w,x){p}_{n}(w,y)-p_n(w,x) {p}_{n+1}(w,y)}{x-y}= 
\frac{p_{n+1}(x){p}_{n}(y)-p_n(x) {p}_{n+1}(y)}{c_n(x-y)},
\end{equation} 
being $p_n(w,x)=p_n(x)/\sqrt{c_n}$.
By combining \eqref{Ktilde} with the second identity of \eqref{eq:Kn2}, we obtain
$$2 c_n(x-y) \tilde{K}_n(w,x,y)=2\bigl[p_{n+1}(x)p_n(y)-p_n(x)p_{n+1}(y)\bigr]-(x-y)p_n(x)p_n(y). $$
At this point, the assertion follows noting that the last term can be written as
$$(x-y)p_n(x)p_n(y)=\bigl[(a_n-x)-(a_n-y)\bigr]p_n(x)p_n(y),$$
and using \eqref{eq:def-ptilde}. 
\end{proof}
In order to state our first main theorem, let us recall the Christoffel numbers of the anti-Gauss formula \cite{Notaris}
\begin{equation}\label{eq:lambdatilde}
\tilde{\lambda}_k:=\frac{2c_n}{p_n(\tilde{x}_k)\,\tilde{p}_{n+1}'(\tilde{x}_k)},
\qquad k=1,\dots,n+1.
\end{equation}
\begin{theorem}\label{teo:darb}
The following identities hold true for $n \geq 0$ and $k=1,\dots,n+1,$
\begin{equation}\label{eq:ellk-kernel}
\tilde{\ell}_k(x)=\tilde{\lambda}_k\,\tilde{K}_n(w,x,\tilde{x}_k), \qquad  
\tilde{\lambda}_k= \left[\tilde{K}_n(w,\tilde{x}_k,\tilde{x}_k)\right]^{-1}, \qquad \textrm{and} \qquad \tilde{\ell}_k(x)=
\frac{\tilde{K}_n(w,x,\tilde{x}_k)}{\tilde{K}_n(w,\tilde{x}_k,\tilde{x}_k)}. 
\end{equation} 
Consequently, the anti-Gauss Lagrange polynomial interpolating a continuous function $f$ defined in $[a,b]$ is given by
\begin{equation}\label{eq:LmK}
\tilde{L}_{n+1}(f,x)=\sum_{k=1}^{n+1} \tilde{\lambda}_k \tilde{K}_n(w,x,\tilde{x}_k) f(\tilde{x}_k).
\end{equation}
\end{theorem}

\begin{proof}
Let us prove \eqref{eq:ellk-kernel}. An evaluation of \eqref{eq:Ktilde} at $y=\tilde{x}_k$ and considering that $\tilde{p}_{n+1}(\tilde{x}_k)=0$ yields
\[
\tilde{K}_n(w,x,\tilde{x}_k)
=
\frac{p_n(\tilde{x}_k)\tilde{p}_{n+1}(x)}{2c_n(x-\tilde{x}_k)}.
\]
Therefore, by \eqref{eq:lambdatilde} and \eqref{eq:lagrange}, we have  for $x\neq \tilde{x}_k$ that
$
\tilde{\lambda}_k\,\tilde{K}_n(w,x,\tilde{x}_k)
=
\dfrac{\tilde{p}_{n+1}(x)}{\tilde{p}_{n+1}'(\tilde{x}_k)(x-\tilde{x}_k)}
=\tilde{\ell}_k(x),
$
which proves the first thesis.
The second one immediately follows by evaluating the above expression at $x=\tilde{x}_j$ and using $\tilde\ell_k(\tilde x_k)=1$. The last identity of \eqref{eq:ellk-kernel} follows by combining the first two. Relation \eqref{eq:LmK} is a consequence of the first assertion.
\end{proof}

\section{The Barycentric formulation of the anti-Gauss Lagrange polynomial}
\label{sec:barycentric}
In this section, we give a barycentric formulation for the anti-Gauss Lagrange polynomial. This representation inherits essentially the same structure as the well-known barycentric form available in the literature \cite{berrut2004,higham2004}.
\begin{theorem}
The anti-Gauss Lagrange interpolant can be written in the following barycentric form for a continuous function $f$ defined in $[a,b]$
\begin{equation}\label{eq:barycentric}
\tilde{L}_{n+1}(f,x)
=
\frac{\displaystyle\sum_{k=1}^{n+1}\frac{\omega_k}{x-\tilde{x}_k}f(\tilde{x}_k)}{\displaystyle\sum_{k=1}^{n+1}\frac{\omega_k}{x-\tilde{x}_k}},
\qquad x\neq \tilde{x}_k, \qquad \omega_k:=\dfrac{1}{\tilde{p}'_{n+1}(\tilde{x}_k)}.
\end{equation} 
\end{theorem}
\begin{proof}
By \eqref{eq:lagrange} we have
$
\tilde{L}_{n+1}(f,x)=\displaystyle \tilde{p}_{n+1}(x) \sum_{k=1}^{n+1} \frac{\omega_k}{x-\tilde{x}_k} f(\tilde{x}_k).
$ 
At this point \eqref{eq:barycentric} follows by considering that \\
$\displaystyle 1=\sum_{k=1}^{n+1}\tilde{\ell}_k(x)= \tilde{p}_{n+1}(x) \sum_{k=1}^{n+1} \frac{\omega_k}{(x-\tilde{x}_k)}.$
\end{proof}
\begin{remark}
Let us note that the barycentric weights $\omega_k$ involve the first derivative of $\tilde{p}_{n+1}$. Howevere, this computation can be avoid by using \eqref{eq:lambdatilde} getting $ \displaystyle
\omega_k=
\frac{\tilde{\lambda}_k\,p_n(\tilde{x}_k)}{2c_n}. $
\end{remark}

\section{The orthogonal expansion of the anti-Gauss Lagrange interpolant} \label{sec:modal}
In this section, we provide an orthogonal expansion of the anti-Gauss Lagrange interpolating polynomial. To this end, let us introduce the following modified orthonormal family
\begin{equation*}
\phi_j(x)=\begin{cases}
p_j(w,x), & j=0,\dots,n-1, \\
\frac{1}{\sqrt{2}} p_n(w,x), & j=n,
\end{cases}
\end{equation*} 
and note that the modified Darboux kernel introduced in Definition \ref{def:modK} can be expressed as
\begin{equation}\label{eq:Kn3}
\tilde{K}_n(w,x,y)= \sum_{j=0}^n \phi_j(x) \phi_j(y) = \boldsymbol{\phi}(x) \boldsymbol{\phi}(y)^T , \qquad \boldsymbol{\phi}(x)=[\phi_0(x),\ldots,\phi_n(x)].
\end{equation}
\begin{proposition}\label{prop:phi}
The family $\{\phi_0,\dots,\phi_n\}$ is orthonormal with respect to the discrete anti-Gauss inner product
\[
\langle f,g\rangle_{\mathrm{AG}}:=\sum_{k=1}^{n+1}\tilde{\lambda}_k f(\tilde{x}_k)g(\tilde{x}_k),
\]
namely, $\langle \phi_i,\phi_j\rangle_{\mathrm{AG}}=\delta_{ij}$, for $0\le i,j\le n$, where $\delta_{ij}$ stands for the Dirac delta.
\end{proposition}
\begin{proof}
Firstly note that by combining \eqref{eq:ellk-kernel} with \eqref{eq:Kn3},  we have
\begin{equation}\label{eq:kernel-expansion-AG}
\tilde{\ell}_k(x)=\tilde{\lambda}_k\sum_{j=0}^{n}\phi_j(x)\phi_j(\tilde{x}_k), \qquad k=1,\ldots,n+1.
\end{equation}
For each fixed $i=0,\ldots,n$, since $\phi_i\in \mathbb{P
}_n$, the Lagrange interpolation formula gives
$ \displaystyle 
\phi_i(x)=\sum_{k=1}^{n+1}\tilde{\ell}_k(x)\phi_i(\tilde{x}_k),
$
from which substituting \eqref{eq:kernel-expansion-AG} yields
\[
\phi_i(x)=\sum_{j=0}^{n}\phi_j(x)\left(\sum_{k=1}^{n+1}\tilde{\lambda}_k\phi_j(\tilde{x}_k)\phi_i(\tilde{x}_k)\right)=\sum_{j=0}^{n}\phi_j(x) \langle \phi_i,\phi_j\rangle_{\mathrm{AG}} , \qquad i=0,\ldots,n.
\]
Then, a comparison of the coefficients furnishes the thesis.
\end{proof}
Note that the above proposition allow us to give also an alternative proof to the second identity of \eqref{eq:ellk-kernel}. In fact, let us define the square matrix $Q$ of order $n+1$ as
\begin{equation}\label{eq:def-Q-AG}
Q_{k,j}:=\sqrt{\tilde{\lambda}_k}\,\phi_j(\tilde{x}_k), \quad k=1,\dots,n+1, \quad j=0,\ldots,n.
\end{equation}
Note that by Proposition \ref{prop:phi}, it easily follows that $(Q^TQ)_{ij}=\delta_{ij}$, for  $0 \leq i,j \leq n$, from which $Q^TQ=I_{n+1}=QQ^T$, where $I_{n+1}$ represents the identity matrix of order $n+1$. Hence, $Q$ is an orthogonal matrix. Moreover, 
\[
(QQ^T)_{ik}
=\sum_{j=0}^{n}Q_{i,j}Q_{k,j}
=\sqrt{\tilde{\lambda}_i\tilde{\lambda}_k}\sum_{j=0}^{n}\phi_j(\tilde{x}_i)\phi_j(\tilde{x}_k)
=\sqrt{\tilde{\lambda}_i\tilde{\lambda}_k}\,\tilde{K}_n(w,\tilde{x}_i,\tilde{x}_k), \qquad \forall i,k=1,\ldots,n+1.
\]
Then, using $QQ^T=I_{n+1}$, we obtain 
\begin{equation}\label{eq:nodal-kernel-AG}
\tilde{K}_n(w,\tilde{x}_i,\tilde{x}_k)=\frac{\delta_{ik}}{\sqrt{\tilde{\lambda}_k \tilde{\lambda}_i}},
\qquad 1\le i,k\le n+1.
\end{equation}
\begin{theorem}
The anti-Gauss Lagrange interpolant of a function $f$ continuous in $[a,b]$ admits the orthogonal expansion
\begin{equation}\label{eq:modal-form-AG}
\tilde{L}_{n+1}(f,x)=\sum_{j=0}^{n} c_j\phi_j(x), \qquad c_j:=\sum_{k=1}^{n+1}\tilde{\lambda}_k f(\tilde{x}_k)\phi_j(\tilde{x}_k)=Q^T \tilde{\Lambda}^{1/2} {\bf{f}},
\qquad j=0,\dots,n,
\end{equation}
where $\tilde{\Lambda}:=\operatorname{diag}(\tilde{\lambda}_1,\dots,\tilde{\lambda}_{n+1})$,  ${\bf{f}}=[f(\tilde{x}_1), \ldots, f(\tilde{x}_{n+1})]^T$, and  $Q$ is given in \eqref{eq:def-Q-AG}.
\end{theorem}
\begin{proof}
Let $f$ be a function, and define
$
\displaystyle P_f(x):=\sum_{j=0}^{n} c_j\phi_j(x),
$
with the coefficients $c_j$ given in \eqref{eq:modal-form-AG}. Then, for every $i=1,\dots,n+1$, we have
\begin{align*}
P_f(\tilde{x}_i)
=\sum_{j=0}^{n} \left[\sum_{k=1}^{n+1}\tilde{\lambda}_k f(\tilde{x}_k)\phi_j(\tilde{x}_k)\right]\phi_j(\tilde{x}_i)=\sum_{k=1}^{n+1}\tilde{\lambda}_k f(\tilde{x}_k) \left[\sum_{j=0}^{n} \phi_j(\tilde{x}_i) \phi_j(\tilde{x}_k) \right]= \sum_{k=1}^{n+1}\tilde{\lambda}_k f(\tilde{x}_k)\tilde{K}_n(w,\tilde{x}_i,\tilde{x}_k),
\end{align*}
where the last identity is due to \eqref{eq:Kn3}. Now, by virtue of \eqref{eq:nodal-kernel-AG}, we deduce
$
P_f(\tilde{x}_i)=\displaystyle \sum_{k=1}^{n+1}\tilde{\lambda}_k f(\tilde{x}_k)\frac{\delta_{ik}}{\sqrt{\tilde{\lambda}_k\tilde{\lambda}_i}}
=f(\tilde{x}_i).
$
Thus, $P_f\in\mathbb P_n$ interpolates $f$ at the $n+1$ distinct nodes $\tilde{x}_1,\dots,\tilde{x}_{n+1}$. By uniqueness of polynomial interpolation, we have $
P_f(x)=\tilde{L}_{n+1}(f,x)$, which proves \eqref{eq:modal-form-AG}.
\end{proof}

\section{A numerical test}\label{sec:tests}
The aim of this numerical test is to give a comparison 
of the proposed formulations in terms of stability.
In \eqref{eq:lagrange}, the fundamental polynomials $\tilde\ell_k(x)$ have a denominator that may become very small when $\tilde x_k$ is very close to $\tilde x_j$. Hence, for $n$ sufficiently large, we may have overflow in the computation of $\tilde\ell_k(x)$ for some $k$. Analogously, underflow can occur when the denominator is large and we evaluate the fundamental polynomial at a point close to a node $\tilde x_j$. Here, we want to explore whether the other formulations can avoid underflow and overflow problems, thereby enabling the construction of a Lagrange interpolant for sufficiently large $n$. This is crucial for approximating non-smooth functions with a certain level of accuracy.

Let us consider the anti-Gauss polynomial \eqref{eq:def-ptilde} defined in $[-1,1]$ associated to the monic Chebyshev orthogonal polynomials $p_n(x)= 2^{1-n} \cos{(n \arccos(x))}$, w.r.t. the weight $w(x)=(1-x^2)^{-1/2}$. In \cite[Theorem 2]{DFR2020} it has been proved that
$$\tilde{p}_{n+1}(x)=2^{1-n} U_{n-1}(x) (x^2-1), \qquad U_{n-1}(x)=\frac{\sin{(n \arccos{x})}}{\sin({\arccos{x}})}, \qquad n \geq 1.$$
Therefore, the zeros of such a  polynomial are
$\tilde{x}_k=\cos{\left((n-k+1)\frac{\pi}{n}\right)}$,  $k=1,\ldots,n+1$. Note that this set is an optimal set of nodes for the convergence \cite{MO2001}.
Let us now consider the Lagrange polynomial interpolating the functions $f_1(x)=|x|+\frac{x}{2}+x^2$ and $f_2(x)=e^{|x-0.5|^{5/2}}$ at the zeros $\tilde{x}_k$. Table \ref{table1} reports the maximum absolute error at $100$ equidistant points in $(-1,1)$, by using the four representations
$$\epsilon_i(f_j)=\|f_j-\tilde{L}_{n+1}(f_j)\|, \qquad i=\{N,K,B,O\}, \quad j=1,2,$$
where $\epsilon_N$ means that $\tilde{L}_{n+1}$ is computed using the nodal representation \eqref{eq:lagrange}; $\epsilon_K$ refer to the representation \eqref{eq:LmK}; whereas $\epsilon_B$ and $\epsilon_O$ concerns the barycentric form \eqref{eq:barycentric} and the orthogonal expansion \eqref{eq:modal-form-AG}, respectively. As we can see by Table \ref{table1}, the three developed forms are more stable than the nodal representation. Indeed, they allow the computation of the anti-Gauss Lagrange polynomial for large $n$, providing a more accurate approximation of the function $f$. 

We remark that in this specific case, thanks to \cite[Theorem 2]{DFR2020}, the barycentric weights are given by 
$$\omega_1=2^{n-2}/n \, (-1)^k  
, \qquad \omega_k=2^{n-1}/n \, \cos{(\pi(n-k+1))}, \quad k=2,\ldots,n, \qquad \omega_{n+1}= 2^{n-2}/n.$$
and the factor $2^{n-1}/n$ can be dropped since it appears both in the numerator and denominator of \eqref{eq:barycentric}. Without this simplification, the formula is far less stable. Indeed, in Table \ref{table1} we obtain NaN for $n=1024$ and $2048$.

\begin{table}
\centering
\small
\setlength{\tabcolsep}{3.5pt}
\renewcommand{\arraystretch}{1.12}
\caption{Abosule errors for the functions $f_j$ $j=1,2$}\label{table1}
\begin{tabular}{@{}rcccccccc@{}}
\toprule
 $n$ & $\epsilon_N(f_1)$ & $\epsilon_K(f_1)$ & $\epsilon_B(f_1)$ & $\epsilon_O(f_1)$ & $\epsilon_N(f_2)$ & $\epsilon_K(f_2)$ & $\epsilon_B(f_2)$ & $\epsilon_O(f_2)$  \\ 
\midrule
 
32 	& 1.77e-02 	 & 1.77e-02 	 & 1.77e-02 	 & 1.77e-02 & 9.95e-05 	 & 9.95e-05 	 & 9.95e-05 	 & 9.95e-05 \\ 
  64 	& 7.59e-03 	 & 7.59e-03 	 & 7.59e-03 	 & 7.59e-03  & 1.47e-05 	 & 1.47e-05 	 & 1.47e-05 	 & 1.47e-05\\ 
 128 	& 4.66e-03 	 & 4.66e-03 	 & 4.66e-03 	 & 4.66e-03 & 1.35e-06 	 & 1.35e-06 	 & 1.35e-06 	 & 1.35e-06 \\ 
 256 	& 9.45e-04 	 & 9.45e-04 	 & 9.45e-04 	 & 9.45e-04 & 3.82e-07 	 & 3.82e-07 	 & 3.82e-07 	 & 3.82e-07  \\ 
 512 	& 4.83e-04 	 & 4.83e-04 	 & 4.83e-04 	 & 4.83e-04  & 4.09e-08 	 & 4.09e-08 	 & 4.09e-08 	 & 4.09e-08\\ 
 1024 	&      NaN 	 & 1.05e-04 	 & 1.05e-04 	 & 1.05e-04  &      NaN 	 & 3.37e-09 	 & 3.37e-09 	 & 3.37e-09\\ 
 2048 	&      NaN 	 & 3.70e-05 	 & 3.70e-05 	 & 3.70e-05 &      NaN 	 & 1.39e-10 	 & 1.40e-10 	 & 1.39e-10 \\ 
\bottomrule
\end{tabular}
\end{table}
 
\section*{Acknowledgements}
The authors are members of the GNCS group. P. D\'iaz de Alba and L. Fermo are partially supported by INdAM-GNCS 2026 project ``Metodi numerici per modelli integrali e dinamiche con memoria''. L. Fermo is also partially supported by Fondazione di Sardegna biennial project 2024-2025 ``Integral and Discrete Inverse Problems (InDIP)''. V. Loi is partially supported by INdAM-GNCS 2026 project ``Metodi strutturati per il signal processing avanzato''.

\bibliographystyle{plain}

\bibliography{biblio}

\end{document}